\documentclass[11pt, reqno]{amsart}

\usepackage{amsmath}
\usepackage{amssymb}
\usepackage{amsfonts}
\usepackage{amsbsy}
\usepackage{mathrsfs}
\usepackage{relsize}

\usepackage{dsfont}
\usepackage{mathtools}
\mathtoolsset{showonlyrefs}
\usepackage{enumitem}

\usepackage{enumitem}
\usepackage{xcolor}
\usepackage{tikz}

\usepackage[
  left=2cm,
  right=2cm
]{geometry}

\numberwithin{equation}{section}
\usepackage[scaled]{helvet}

\newtheorem{theorem}{Theorem}[section]
\newtheorem{lemma}[theorem]{Lemma}
\newtheorem{proposition}[theorem]{Proposition}

\newtheoremstyle{remarkstyle}
{}{}{}{ }{\bfseries}{.}{ }{\thmname{#1}\thmnumber{ #2}\thmnote{ (#3)}}
\theoremstyle{remarkstyle}
\newtheorem{remark}{Remark}[section]
\newtheorem{definition}{Definition}[section]

\newcommand{\R}{\mathbb R}
\newcommand{\C}{\mathbb C}

\definecolor{lime}{HTML}{A6CE39}
\DeclareRobustCommand{\orcidicon}{
	\begin{tikzpicture}
	\draw[lime, fill=lime] (0,0)
	circle [radius=0.16]
	node[white] {{\fontfamily{qag}\selectfont \tiny ID}};
	\draw[white, fill=white] (-0.0625,0.095)
	circle [radius=0.007];
	\end{tikzpicture}
	\hspace{-2mm}
}
\foreach \x in {A, B, C}{\expandafter\xdef\csname orcid\x\endcsname{\noexpand\href{https://orcid.org/\csname orcidauthor\x\endcsname}
			{\noexpand\orcidicon}}
}

\usepackage[colorlinks, linkcolor=venetianred, citecolor=blue, urlcolor=Cblue, hypertexnames=false]{hyperref}
\definecolor{Cblue}{rgb}{0.50,0.85,0.85}
\definecolor{venetianred}{rgb}{0.78, 0.03, 0.08}

\title[Inhomogeneous NLS]
{Normalized Standing Waves for the Focusing Inhomogeneous Schr\"odinger Equation
with Spatially Growing Nonlinearity
}

\author[Mohamed Majdoub]{Mohamed Majdoub\orcidB{}}
\author[Tarek Saanouni]{Tarek Saanouni\orcidC{}}

\address[M. Majdoub]{Department of Mathematics, College of Science, Imam Abdulrahman Bin Faisal University, P. O. Box 1982, Dammam, Saudi Arabia.\newline
	Basic and Applied Scientific Research Center, Imam Abdulrahman Bin Faisal University, P.O. Box 1982, 31441, Dammam, Saudi Arabia.}
\email{\sl\textcolor{cyan}{mmajdoub@iau.edu.sa}}
\email{\sl\textcolor{cyan}{med.majdoub@gmail.com}}
\address[T. Saanouni]{Department of Mathematics, College of Science, Qassim University, Buraydah, Kingdom of Saudi Arabia.\newline
	University of Tunis El Manar, Faculty of Sciences of Tunis, LR03ES04 partial differential equations, 2092 Tunis, Tunisia.}
\email{\sl\textcolor{cyan}{t.saanouni@qu.edu.sa}}
\email{\sl\textcolor{cyan}{tarek.saanouni@ipeiem.rnu.tn}}


\begin{document}
	
\begin{abstract}
We study the focusing inhomogeneous nonlinear Schr\"odinger equation
$$
i\partial_t u + \Delta u = -|x|^b |u|^{p-1}u \quad \text{in } \mathbb{R}^N,
$$
with $b>0$ and $p>1$. Due to the spatial growth of the nonlinearity, standard
compactness arguments do not apply and new difficulties arise.

We first characterize ground state standing waves via a variational approach
on the Nehari manifold and we establish some sharp stability and instability properties.
In the $L^2$-subcritical regime, we prove the existence of normalized ground
states by solving a constrained energy minimization problem in the radial
energy space, and we show that the resulting set of minimizers is orbitally
stable under the flow. In contrast, in the $L^2$-critical and supercritical
regimes, ground state standing waves are shown to be strongly unstable by
finite-time blow-up.

Our results extend classical stability and instability theory for nonlinear
Schr\"odinger equations to the case of spatially growing inhomogeneous
nonlinearities.
\end{abstract}

        \date{\today}
	\maketitle	

    \section{Introduction}
\label{S1}
\setcounter{equation}{0}

In this paper, we consider the following focusing inhomogeneous nonlinear
Schr\"odinger equation
\begin{align} \label{INLS}
i \partial_t u + \Delta u = - |x|^b |u|^{p-1} u, \quad (t,x) \in \R_+ \times \R^N,
\end{align}
where $u=u(t,x)$ is a complex-valued function, $b\in\R$, and $p>1$.

Equation~\eqref{INLS} is a particular case of a broader class of
inhomogeneous nonlinear Schrödinger equations in which the spatial weight
$-|x|^b$ is replaced by a general function $K(x)$. Such models naturally
arise in several physical contexts where the medium is spatially inhomogeneous.
In nonlinear optics, for example, spatially varying refractive indices or
inhomogeneous gain/loss profiles lead to effective nonlinear Schrödinger
equations with spatial weights representing nonuniform media
\cite{Gill,GS}. Related models appear in the study of
optical beam propagation in nonuniform plasmas \cite{Gill} and in
Bose–Einstein condensates under spatially dependent interactions
\cite{Strauss,GS}. From a mathematical point of view, the presence of a spatially dependent coefficient
in the nonlinear term profoundly modifies both the variational structure and the
qualitative behavior of solutions. In particular, the weighted nonlinearity affects
the existence and stability of standing waves and plays a decisive role in the long-time
dynamics, which has motivated a substantial literature on well-posedness, scattering,
and blow-up for inhomogeneous nonlinear Schrödinger equations
\cite{Cazenave,CG-DCDS-B,Farah}. Moreover, several classical tools available in the
homogeneous case; such as arguments based on translation invariance or uniform coercivity;
are no longer applicable, and the analysis becomes strongly dependent on the behavior
of the weight at infinity.

The Cauchy problem for \eqref{INLS} with $b<0$, corresponding to spatially decaying
nonlinearities, has attracted considerable attention over the past decades (see,
for instance, \cite{GS, Guzman, Dinh, AT, KLS, Dinh-NA, Farah, BL, FG-JDE,
FG-BBMS, Campos, Dinh-2D, MMZ, CFGM, DK-SIAM, Murphy} and the references therein).
In this regime, the decay of $|x|^b$ allows one to exploit weighted Sobolev
embeddings and concentration--compactness techniques to study well-posedness,
scattering, and blow-up phenomena. We also refer to \cite{CPAM2021} for recent
results on blow-up and strong instability of ground states for \eqref{INLS}.

The case $b>0$, which corresponds to spatially growing nonlinearities, is
substantially more delicate. Indeed, the growth of the weight $|x|^{b}$ at infinity
prevents a direct control of the potential energy by means of standard Sobolev
embeddings or Hardy-type inequalities in the energy space $H^{1}(\R^{N})$.
As a result, both the local and global analysis of \eqref{INLS} require refined
estimates that are adapted to the weighted structure of the nonlinearity and,
in particular, to the radial framework.

Equation~\eqref{INLS} enjoys the following scaling invariance:
$$
u_{\lambda}(t,x)\coloneqq \lambda^{\frac{2 + b}{p - 1}}u(\lambda^{2}t,\lambda x),
\quad \lambda >0.
$$
A direct computation yields
$$
\| u_{\lambda}(0)\|_{\dot{H}^{s}}
= \lambda^{s + \frac{2 + b}{p - 1} - \frac{N}{2}}\| u_{0}\|_{\dot{H}^{s}},
$$
so that the scaling critical regularity is $\dot{H}^{s_{c}}$, where
$$
s_{c}\coloneqq \frac{N}{2} -\frac{2 + b}{p - 1}.
$$
We denote by
\begin{equation}
    \label{pc}
    p_c := 1 + \frac{4+2b}{N}
\end{equation}
the mass-critical exponent, and by
\begin{equation}
    \label{pcc}
    p^c :=
\begin{cases}
1 + \frac{4+2b}{N-2}, & N\ge 3,\\[6pt]
\infty, & N=1,2.
\end{cases}
\end{equation}
The case $s_c=0$ corresponds to $p=p_c$ (mass-critical), while $s_c=1$
corresponds to $p=p^c$ when $N\ge 3$ (energy-critical). When $0<s_c<1$, that is,
when $p_c<p<p^c$ for $N\ge 3$, we say that \eqref{INLS} is intercritical
(mass-supercritical and energy-subcritical).

Using a Gagliardo--Nirenberg type inequality together with energy methods (see
\cite[Theorem~3.3.9, p.~71]{Cazenave}), Chen and Guo
\cite{CG-DCDS-B, Chen, Chen-CMJ, CG-AM} proved local well-posedness in
$H^1_{\rm rad}(\R^N)$ under the assumptions
\begin{align} \label{cond-CG}
N\geq 2, \quad b>0, \quad p>1+\frac{2b}{N-1}, \quad
p<\frac{N+2}{N-2}+\frac{2b}{N-1} \text{ if } N\geq 3.
\end{align}

As later observed in \cite{DMS2021}, the above theory leaves uncovered the range
$$
\frac{N+2}{N-2}+\frac{2b}{N-1} \leq p \leq \frac{N+2+2b}{N-2},
$$
which is not addressed in
\cite{CG-DCDS-B, Chen, Chen-CMJ, CG-AM}. This gap was subsequently filled in
\cite{DMS2021}, where the authors also developed a long-time dynamical theory for
\eqref{INLS}, including global existence, energy scattering, and finite-time
blow-up.

The main purpose of the present paper is to investigate the existence and the
stability/instability properties of ground state standing waves for \eqref{INLS}
in the case $b>0$.
 By a standing wave, we mean a solution of
\eqref{INLS} of the form
$$
u(t,x)= {\rm e}^{i \omega t}\phi(x),
$$
where $\omega\in\R$ is a real frequency and
$\phi\in H^1_{\rm rad}(\R^N)$ is a nontrivial solution of the stationary equation
\begin{equation}
\label{GS}
-\Delta \phi+\omega\phi=|x|^b\,|\phi|^{p-1}\phi.
\end{equation}

	Define the following $C^2$ functionals on $H^1_{\rm rad}(\R^N)$:
	\begin{eqnarray*}
	 S_\omega(\phi)&=&\frac{1}{2}\int_{\R^N}\, |\nabla\phi(x)|^2 dx+\frac{\omega}{2}\int_{\R^N}\, |\phi(x)|^2 dx-\frac{1}{p+1}\int_{\R^N}\,|x|^b|\phi(x)|^{p+1} dx,\\
	 I_{\omega}(\phi)&=&\int_{\R^N}\, |\nabla\phi(x)|^2 dx+\omega\int_{\R^N}\, |\phi(x)|^2 dx-\int_{\R^N}\,|x|^b|\phi(x)|^{p+1} dx,\\
	 P(\phi)&=&\int_{\R^N}\, |\nabla\phi(x)|^2 dx-\frac{N(p-1)-2b}{2(p+1)}\int_{\R^N}\,|x|^b|\phi(x)|^{p+1} dx.
	\end{eqnarray*}
	Note that if $\omega\leq 0$ then \eqref{GS} has no nontrivial solution, see Remark \ref{rm1}. For $\omega >0$, it was shown in \cite[Proposition 2.3.]{DMS2021} that \eqref{GS} has a unique positive radial solution $Q_\omega\in H^1_{\rm rad}(\R^N)$ if $N\geq 2, b>0$ and $1+\frac{2b}{N-1}<p<\frac{N+2+2b}{N-2}.$
	
	Let $\mathcal{A}_{\omega}$ denote the set of nontrivial solutions of \eqref{GS}, namely
	\begin{equation}
	 \label{A}
	 \mathcal{A}_{\omega}=\bigg\{\, \phi\in H^1_{\rm rad}(\R^N)\backslash\{0\}:\;\;\; \phi \;\;\;\mbox{solves}\;\;\;\eqref{GS}\,\bigg\}.
	\end{equation}
	From \cite{Sin-Will} we know that $\mathcal{A}_{\omega}\neq\emptyset$ for $1+\frac{2b}{N-1}<p<\frac{N+2+2b}{N-2}.$ The set of ground states is defined as
	\begin{equation}
	\label{GSS}
	\mathcal{G}_{\omega}=\bigg\{\, \phi\in \mathcal{A}_{\omega}:\;\;\; S_{\omega}(\phi)\leq S_{\omega}(v),\;\; \forall \;\; v\in \mathcal{A}_{\omega}\,\bigg\}.
	\end{equation}
	Define the Nehari manifold 
	\begin{equation}
	   \label{Neh}
	   \mathcal{N}_{\omega}=\bigg\{\, \phi\in H^1_{\rm rad}(\R^N)\backslash\{0\}:\;\;\; I_{\omega}(\phi)=0\,\bigg\}.
	\end{equation}
	Consider the minimizing problem
	\begin{equation}
	\label{domega}
	d(\omega)=\inf\bigg\{\, S_{\omega}(\phi):\;\;\; \phi\in \mathcal{N}_{\omega}\,\bigg\}.
	\end{equation}
	We also define 
	\begin{equation}
	 \label{M}
	 \mathcal{M}_{\omega}=\bigg\{\, \phi\in \mathcal{N}_{\omega}:\;\;\; S_{\omega}(\phi)=d(\omega)\,\bigg\}.
	\end{equation}
	Our first result gives a complete description of $\mathcal{G}_{\omega}$ and can be stated as follows.
	\begin{theorem}
	 \label{G}
	 Let $N\geq 2, b>0, \omega>0$, and $1+\frac{2b}{N-1}<p<p^c$. Then
	 \begin{enumerate}[label=(\roman*)]
	      \item $\mathcal{M}_{\omega}\neq\emptyset$.
	     \item $\mathcal{G}_{\omega}=\mathcal{M}_{\omega}$.
	     \item $\mathcal{G}_{\omega}=\bigg\{\, {\rm e}^{i\theta}\,Q_{\omega},\;\;\; \theta\in\R\,\bigg\}$ where $Q_\omega\in H^1_{\rm rad}(\R^N)$ is the unique positive solution of \eqref{GS}.
	  
	 \end{enumerate}
	\end{theorem}
	\begin{remark}
    \label{rem:11}
\rm 
~\begin{enumerate}[label=(\roman*)]
	      \item Theorem~\ref{G} shows that the variational characterization on the Nehari manifold
is \emph{exactly} the ground state characterization for the stationary problem \eqref{GS}.
In particular, every minimizer of \eqref{domega} is a solution of \eqref{GS}. and conversely,
every ground state must minimize the action on $\mathcal N_\omega$.
This identification is crucial for the dynamical arguments in the potential-well analysis.
\item Part~(iii) implies that the ground state set $G_\omega$ consists of a single orbit under
the gauge symmetry $u\mapsto e^{i\theta}u$.
This reduces the stability/instability analysis to perturbations of a unique positive
radial profile $Q_\omega$.
\item The radial assumption is essential in the present setting $b>0$,
since the weight $|x|^b$ destroys translation invariance and prevents the use of the
usual concentration--compactness arguments in $H^1(\R^N)$.
Compactness is recovered here through the compact embedding
$H^1_{\mathrm{rad}}(\R^N)\hookrightarrow L^{p+1}(\R^N,|x|^b\,dx)$.
\end{enumerate}
\end{remark}
For $\omega>0$, we define the following subsets of $H^1(\R^N)$:
	\begin{eqnarray*}
	\mathcal{K}^+_{\omega}&=&\bigg\{\, \phi\in H^1_{\rm rad}(\R^N):\; \;S_{\omega}(\phi)<S_{\omega}(Q_{\omega})\;\;\mbox{and}\;\; P(\phi)\geq 0\,\bigg\},\\
	\mathcal{K}^-_{\omega}&=&\bigg\{\, \phi\in H^1_{\rm rad}(\R^N):\;\; S_{\omega}(\phi)<S_{\omega}(Q_{\omega})\;\;\mbox{and}\;\; P(\phi)< 0\,\bigg\}.
	\end{eqnarray*}
    We now establish a sharp dichotomy between global existence and finite-time blow-up
for energy solutions of \eqref{INLS}, in the spirit of the potential-well theory
introduced by Payne and Sattinger~\cite{ps}.
\begin{theorem}
\label{Kpm}
Let $N\geq 2, b>0$ and $\omega>0$.
	 ~\begin{enumerate}[label=(\roman*)]
	      \item The sets $\mathcal{K}^+_{\omega}$ and $\mathcal{K}^-_{\omega}$ are invariant under the flow generated by \eqref{INLS}.
	     \item If $u_0\in \mathcal{K}^+_{\omega}$, then the corresponding maximal solution $u$ exists globally in time.
	     \item If $p_c<p<p^c$ and $u_0\in \mathcal{K}^-_{\omega}\cap \Sigma$, then the corresponding maximal solution $u$ blows up in finite time. Here $\Sigma= H^1\cap L^2(|x|^2\,dx)$.
	   \end{enumerate}
\end{theorem}
\begin{remark}
\rm 
~\begin{enumerate}[label=(\roman*)]
	      \item Theorem~\ref{Kpm} is a potential-well dichotomy in the spirit of Payne--Sattinger~\cite{ps}.
The functionals $S_\omega$ and $P$ play the role of a Lyapunov threshold and a
virial-type sign functional, respectively.
The invariant decomposition $\mathcal{K}_\omega^+\cup \mathcal{K}_\omega^-$ gives a sharp separation
between global existence and blow-up below the ground state level.
\item In Part~(iii), the assumption $u_0\in \Sigma$ is used to justify the virial identity
and obtain finite-time blow-up.
Since the solutions considered are radial, the condition can often be relaxed using localized virial estimates.
In particular, we also have blow-up for $u_0\notin\Sigma$ under the extra assumption $p\le 5$.

\item The restriction $p>p_c$ in the blow-up statement corresponds to the mass-supercritical
regime, where the virial functional has the correct sign structure.
In the mass-subcritical regime $p<p_c$, the same argument cannot yield blow-up,
and indeed the dynamics are dominated by global existence and stability phenomena.
\end{enumerate}
\end{remark}

Next, we give the definition of stability/instability of standing waves.
\begin{definition}
\label{Stab-Instab}
Let $\omega>0$.
~\begin{enumerate}[label=(\roman*)]
	      \item A standing wave ${\rm e}^{iwt}\phi(x)$ is said to be (orbitally) stable if for any $\varepsilon>0$ there exists $\eta>0$ such that if $u_0\in H^1$ satisfies $\|u_0-\phi\|<\eta$, then the corresponding maximal solution $u(t)$ to \eqref{INLS} exists globally in time with 
    $$
    \sup_{t\in\R}\,\inf_{\theta\in\R}\,\|u(t)-{\rm e}^{i\theta}\phi\|_{H^1}<\varepsilon.
    $$
\item A standing wave ${\rm e}^{iwt}\phi(x)$ is said to be unstable if it is not stable.
\item A standing wave ${\rm e}^{iwt}\phi(x)$ is said to be strongly unstable if for any $\varepsilon>0$ there exists $u_0\in H^1$ such that $\|u_0-\phi\|<\varepsilon$ and the corresponding maximal solution $u(t)$ to \eqref{INLS} blows up in finite time.
\end{enumerate}
\end{definition}
\begin{remark}
\rm
The stability notion in Definition~\ref{Stab-Instab} is the classical \emph{orbital stability}
in $H^1(\R^N)$, which is the natural energy space associated with \eqref{INLS}.
The minimization over the phase parameter,
$$
\inf_{\theta\in\R}\|u(t)-e^{i\theta}\phi\|_{H^1},
$$
accounts for the gauge invariance of \eqref{INLS} under the transformation
$u\mapsto e^{i\theta}u$.
\end{remark}

As an immediate consequence of Theorem~\ref{Kpm}, we derive the following
stability/instability criterion.
\begin{theorem}
\label{Stab-Instab-GS}
Let $N\geq 2, \omega>0,$ and $b>0$.
~\begin{enumerate}[label=(\roman*)]
    \item If $p_c< p<p^c$, then for each $\phi\in \mathcal{G}_{\omega}$, the standing wave ${\rm e}^{iwt}\phi(x)$ is strongly unstable.
    \item If $p=p_c$, $N\geq 3$ and $0<b\leq N-2$, then for each $\phi\in \mathcal{G}_{\omega}$, the standing wave ${\rm e}^{iwt}\phi(x)$ is strongly unstable.
    \item If $1+\frac{2b}{N-1}<p<p_c$, then for each $\phi\in \mathcal{G}_{\omega}$, the standing wave ${\rm e}^{iwt}\phi(x)$ is stable.
\end{enumerate}
\end{theorem}
\begin{remark}
\rm 
~\begin{enumerate}[label=(\roman*)]
\item The assumption
$$
1+\frac{2b}{N-1} < p < p_c
$$
implies, in particular, that
$$
N > 1+\frac{b}{2}.
$$
This condition imposes a restriction on the admissible values of $N$ (and consequently on $b$) required to ensure stability.
\item Theorem~\ref{Stab-Instab-GS} provides a sharp stability/instability trichotomy according to the
mass-critical exponent $p_c$.
This extends the classical theory for homogeneous NLS ($b=0$) to the spatially growing
inhomogeneous case $b>0$, where standard compactness and scaling arguments require
significant modifications.
\item The strong instability statements in Parts~(i)--(ii) are proved by constructing
arbitrarily small $H^1$ perturbations of ground states that lie in $\mathcal{K}_\omega^-$,
and then invoking the blow-up mechanism from Theorem~\ref{Kpm}.
This approach is robust and avoids delicate spectral computations.
\end{enumerate}
\end{remark}
We are also interested in \emph{normalized standing waves} for \eqref{INLS}, \emph{i.e.},
standing waves whose mass is prescribed. This question is naturally motivated by the
fact that the mass $\|u(t)\|_{L^2}^2$ is conserved along the flow of \eqref{INLS}.
As a first observation, a scaling argument shows that one can always construct a solution
to \eqref{GS} with a given mass whenever
$$
1+\frac{2b}{N-1}<p<p^c, \qquad p\neq p_c,
$$
where the exponents $p_c$ and $p^c$ are given by \eqref{pc} and \eqref{pcc}, respectively.

\begin{proposition}
\label{scal}
Let $N\geq 2, b>0, $ $1+\frac{2b}{N-1}<p<p^c$ and $p\neq p_c.$ Then for any $c>0$, there exists a solution $(\phi_c, \omega_c)\in H^1_{\rm rad}\times (0,\infty)$ to \eqref{GS} such that
\begin{equation}
    \label{Scal1}
    \|\phi_c\|_{L^2}^2=c.
\end{equation}
Moreover, we have
\begin{equation}
    \label{wc-phic}
    \omega_c = \left(c\| Q\|_{L^{2}}^{-2}\right)^{\frac{2(p - 1)}{N(p_c - p)}},\quad \phi_c(x) = \omega_c^{\frac{2+b}{2(p-1)}} Q\left(\sqrt{\omega_c} x\right),
\end{equation}
where $Q$ is the unique positive radial solution of~\eqref{GS} with $\omega=1$.
\end{proposition}
\begin{remark}
\rm 
~\begin{enumerate}[label=(\roman*)]
\item Proposition~\ref{scal} shows that normalized standing waves for \eqref{INLS} can always be
produced by scaling ground states of \eqref{GS} whenever $p\neq p_c$.
The mass-critical case $p=p_c$ is exceptional because the scaling preserves the $L^2$
norm, so one cannot freely prescribe the mass by dilation.
\item The monotonicity of $\omega_c$ as $c\to 0$ or $c\to\infty$ reflects the competition
between dispersion and the spatially growing focusing nonlinearity.
In particular, in the mass-subcritical regime $p<p_c$, large mass forces large
frequency, while in the mass-supercritical regime $p>p_c$ the behavior is reversed.
\end{enumerate}
\end{remark}

Now, we consider for each $c>0$, 
\begin{equation}
\label{mc}
    \mathbf{m}(c)=\inf\bigg\{\,E(\phi): \;\; \phi\in \mathcal{S}(c)\,\bigg\}\quad\mbox{if}\quad p<p_c,
\end{equation}
and
\begin{equation}
\label{mc1}
\mathbf{m}(c)=\inf\bigg\{\,E(\phi): \;\; \phi\in \mathcal{V}(c)\,\bigg\}\quad\mbox{if}\quad p>p_c,
\end{equation}
where
\begin{equation}
\label{Energ}
E(\phi)=\frac{1}{2}\|\nabla\phi\|_{L^2}^2-\frac{1}{p+1}\int_{\R^N}\,
|x|^b\,|\phi(x)|^{p+1}\,dx,
\end{equation}
\begin{equation}
\label{Sc}
\mathcal{S}(c)=\bigg\{\, \phi\in H^1_{\rm rad}(\R^N):\;\; \|\phi\|_{L^2}^2=c\,\bigg\}.
\end{equation}
and
\begin{equation}
\label{Vc}
\mathcal{V}(c)=\bigg\{\, \phi \in \mathcal{S}(c)\quad\mbox{such that}\quad P(\phi)=0\,\bigg\}.
\end{equation}
Let us denote
\begin{equation}
\label{M-c}
\mathbb{M}_c=\bigg\{\, \phi\in \mathcal{S}(c):\;\;E(\phi)=\mathbf{m}(c)\,\bigg\}\quad\mbox{if}\quad p<p_c, 
\end{equation}
and
\begin{equation}
\label{M-c1}
\mathbb{M}_c=\bigg\{\, \phi\in \mathcal{V}(c):\;\;E(\phi)=\mathbf{m}(c)\,\bigg\}\quad\mbox{if}\quad p>p_c. 
\end{equation}
 We will prove that $\mathbf{m}(c)>-\infty$ in the $L^2$-subcritical regime (see
Section~\ref{S6} below). In contrast, a simple scaling argument shows that
$\mathbf{m}(c)=-\infty$ in the $L^2$-supercritical regime, namely whenever
$$
\max\left(1+\frac{2b}{N-1},\,1+\frac{4+2b}{N}\right)
<p<\frac{N+2+2b}{N-2}.
$$
Indeed, let $\phi\in \mathcal{S}(c)$ and $\mu>0$, and define
$$
\phi_{\mu}(x):=\mu^{N/2}\,\phi(\mu x).
$$
Then $\phi_{\mu}\in \mathcal{S}(c)$ and
$$
E(\phi_{\mu})
=\mu^2\bigg(\frac{1}{2}\|\nabla\phi\|_{L^2}^2
-\frac{\mu^{\kappa}}{p+1}\int_{\R^N}|x|^b|\phi(x)|^{p+1}\,dx\bigg),
$$
where
$$
\kappa:=\frac{N}{2}(p+1)-N-b-2=\frac N2(p-p_c)>0
\qquad\text{since } p>p_c.
$$
Consequently,
$$
E(\phi_{\mu})\to-\infty \quad\mbox{as}\quad \mu\to \infty,
$$
which yields $\mathbf{m}(c)=-\infty$. This divergence justifies the modified
definition of $\mathbf{m}(c)$ in the $L^2$-supercritical regime given in
\eqref{mc1}. See also \cite{Cazenave, Caz-Lions, Luo}.

 \begin{remark}
\rm Note that the assumption
$$
1+\frac{2b}{N-1} < p < 1+\frac{4+2b}{N}
$$
necessarily requires
$$
N > 1+\frac{b}{2}.
$$
Moreover, we have
$$
\max\!\left\{1+\frac{2b}{N-1},\,1+\frac{4+2b}{N}\right\}
=
\left\{
\begin{array}{ll}
1+\frac{4+2b}{N}, & \text{if } N \geq 1+\frac{b}{2}, \\[6pt]
1+\frac{2b}{N-1}, & \text{if } N \leq 1+\frac{b}{2}.
\end{array}
\right.
$$

\end{remark}

We introduce the following notion of orbital stability for normalized standing waves.

\begin{definition}
\label{Orb-Stab-NSW}
Let $c>0$. We say that $\mathbb{M}_c$ is orbitally stable if for any $\epsilon>0$, there exists $\delta>0$ such that for any initial data $u_0$ satisfying
\begin{equation}
\label{Orb-Stab-NSW1}
\inf_{\phi\in \mathbb{M}_c}\,\|u_0-\phi\|_{H^1}<\delta,
\end{equation}
the corresponding solution $u(t)$ to \eqref{INLS} with initial data $u_0$ satisfies
\begin{equation}
\label{Orb-Stab-NSW2}
\inf_{\phi\in \mathbb{M}_c}\,\|u(t)-\phi\|_{H^1}<\epsilon,\quad \forall\;\;t>0.
\end{equation}
\end{definition}
\begin{remark}
\rm
Definition~\ref{Orb-Stab-NSW} provides the natural notion of orbital stability for
normalized standing waves, since the elements of $\mathbb{M}_c$ arise as minimizers of
a variational problem under the mass constraint. The distance is taken with respect to
the whole set $\mathbb{M}_c$, which accommodates the possible lack of uniqueness of
constrained minimizers.
\end{remark}

Our goal is to establish the existence and orbital stability of normalized standing waves
for \eqref{INLS} in the mass-subcritical, mass-critical, and mass-supercritical regimes.
Our first result in this direction is the following.

\begin{theorem}
\label{Norm-SW-Subcrit}
Let $b>0$, $N>1+\frac{b}{2}$, $c>0$, and $1+\frac{2b}{N-1}<p<p_c$. Then, 
~\begin{enumerate}[label=(\roman*)]
\item $\mathbb{M}_c\neq\emptyset$.
    \item $\mathbb{M}_c$ is orbitally stable in the sense of Definition \ref{Orb-Stab-NSW}. 
\end{enumerate}
\end{theorem}
\begin{remark}
\rm
~\begin{enumerate}[label=(\roman*)]
\item The homogeneous case $b=0$ with $p<p_c$ was studied in \cite{Caz-Lions}, where the proof relies on some compactness arguments due to Lions \cite{Lions1, Lions2}.
\item The inhomogeneous case remains less understood. The results of \cite{Lions3} apply, in particular, to the regime $b<0$.
\item To the best of our knowledge, the case $b>0$ has not been previously investigated. It appears to be more delicate, primarily due to the lack of suitable compactness arguments.
\item The embedding \eqref{embed} is clearly not compact in the limiting cases, namely when
$$
p = 1+\frac{2b}{N-1}
\quad \text{or} \quad
p = \frac{N+2+2b}{N-2} \quad \text{if } N \geq 3.
$$
A natural question is whether this lack of compactness can be characterized, even in the homogeneous case $b=0$.
\item Theorem~\ref{Norm-SW-Subcrit} provides a variational construction of normalized ground states
in the radial energy space.
The key point is that in the mass-subcritical regime $p<p_c$ the constrained energy
level $\mathbf{m}(c)$ is finite and negative, which yields compact minimizing sequences and
strong convergence.
\item Orbital stability of $\mathbb{M}_c$ is obtained by a standard contradiction argument:
if stability fails, one constructs a minimizing sequence for $\mathbf{m}(c)$ along the flow,
and then uses compactness of minimizing sequences to recover proximity to $\mathbb{M}_c$.
This is a constrained version of the classical Cazenave--Lions stability method~\cite{Caz-Lions}.
\item The restriction $N>1+\frac b2$ appears naturally in the mass-subcritical range
$1+\frac{2b}{N-1}<p<p_c$, and ensures that the admissible interval of exponents is nonempty.
This is consistent with the dimensional restriction already pointed out in Remark~\ref{rem:11}.
\end{enumerate}

\end{remark}

    This paper is organized as follows. In Section~\ref{S2}, we gather a number of
preliminary tools, including functional inequalities and auxiliary results that
will be used repeatedly in the sequel. Section~\ref{S3} is devoted to the
variational analysis of the stationary problem \eqref{GS}, where we prove the
existence and characterization of ground states and, consequently, establish the
corresponding results for standing waves. In Section~\ref{S4}, we investigate the
dynamical behavior of solutions to \eqref{INLS}, with particular emphasis on the
construction of invariant sets and on the dichotomy between global existence and
finite-time blow-up. Section~\ref{S5} addresses the stability and instability of
standing waves in the $L^2$-subcritical, mass-critical, and $L^2$-supercritical
regimes. In Section~\ref{S6}, we turn to normalized standing waves and study the
constrained variational problem at fixed mass, proving the existence of minimizers
and the orbital stability of the associated set of ground states. Finally,
Section~\ref{sec:concl} contains concluding remarks and a discussion of several
perspectives and open directions suggested by the present work.

	\section{Preliminaries}
\label{S2}
\setcounter{equation}{0}

In this section, we gather several auxiliary results and functional inequalities
that will be  used throughout the paper. For the reader's convenience,
we also recall some basic properties of the energy space and the weighted nonlinear
term appearing in \eqref{INLS}. These preliminary tools provide the analytical
framework for the variational arguments and the dynamical analysis developed in
the subsequent sections.

	\begin{proposition}\cite[Proposition 2.3.]{DMS2021}
	\label{prop-unique}
		Let $N\geq 2, b>0, p>1+\frac{2b}{N-1}$, and $p<\frac{N+2+2b}{N-2}$ if $N\geq 3$. Then there exists a unique positive radial solution to \eqref{GS}.
	\end{proposition}
	\begin{remark}
\rm The uniqueness of the positive radial solution to \eqref{GS} plays a central role:
it implies that the ground state set $G_\omega$ is exactly the gauge orbit of $Q_\omega$.
This avoids any ambiguity in the threshold level $S_\omega(Q_\omega)$.
\end{remark}

	\begin{lemma}\cite[Lemma 2.4.]{DMS2021}
	\label{lem-GN-ineq}
		Let $N\geq 2$, $b>0$, $p \geq 1+\frac{2b}{N-1}$, and $p\leq \frac{N+2+2b}{N-2}$ if $N\geq 3$. Then there exists $C(N,p,b)>0$ such that
		\begin{equation}
		\label{GNI}
		\int_{\R^N} |x|^b |f(x)|^{p+1} dx \leq C(N,p,b) \|\nabla f\|^{\frac{N(p-1)-2b}{2}}_{L^2} \|f\|^{\frac{4+2b-(N-2)(p-1)}{2}}_{L^2}, \quad \forall f \in H^1_{\rm rad}(\R^N).
		\end{equation}
	\end{lemma}

	\begin{lemma}\cite[Lemma 2.7.]{DMS2021} \label{lem-comp-embe}
	Let $N\geq 2$, $b>0$, $p>1+\frac{2b}{N-1}$, and $p<\frac{N+2+2b}{N-2}$ if $N\geq 3$. Then, the embedding
	\begin{equation}
	\label{embed}
	H^1_{\rm rad}(\R^N) \hookrightarrow L^{p+1}_b(\R^N):=L^{p+1}(\R^N, |x|^b\,dx)
	\end{equation}
	is compact. 
	\end{lemma}	
	\begin{remark}
	\label{b=0}
	\rm 
    ~\begin{enumerate}[label=(\roman*)]
\item The case $b=0$ in Lemma \ref{lem-comp-embe} is standard and can be found in several references; see, for instance, \cite{BL, Kav, Strauss}.
\item Lemma~\ref{lem-comp-embe} is the compactness substitute for the missing translation invariance.
In the case $b>0$, the weight $|x|^b$ strongly penalizes mass at infinity,
and radial symmetry allows one to exploit this to recover compactness.
\end{enumerate}
	\end{remark}
	We also recall the following Pohozaev's identities \cite{CG-DCDS-B}.
	\begin{lemma}
	\label{Pohoz}
	Let $N\geq 2$, $b>0$, $p>1+\frac{2b}{N-1}$, and $p<\frac{N+2+2b}{N-2}$ if $N\geq 3$. Let $\phi\in H^1_{\rm rad}(\R^N)$ be a non-trivial solution to \eqref{GS}. Then
	\begin{equation}
	  \label{Pohoz1}
	  \left(\frac{2}{N}-1\right)\int_{\R^N}\,|\nabla\phi(x)|^2\,dx=\omega\int_{\R^N}\,|\phi(x)|^2\,dx
	  -\frac{2(N+b)}{N(p+1)}\,\int_{\R^N}\,|x|^b|\phi(x)|^{p+1}\,dx.
	\end{equation}
	\begin{equation}
	  \label{Pohoz2}
	  \int_{\R^N}\,|\nabla\phi(x)|^2\,dx+\omega\int_{\R^N}\,|\phi(x)|^2\,dx=
	  \int_{\R^N}\,|x|^b|\phi(x)|^{p+1}\,dx.
	\end{equation}
	\end{lemma}
	\begin{remark}\label{rm1}
	\label{Pohoz3}
	\rm 
~\begin{enumerate}[label=(\roman*)]
\item Combining \eqref{Pohoz1} and \eqref{Pohoz2}, we obtain
\begin{equation}
\label{Pohoz4}
\int_{\R^N} |\nabla \phi(x)|^2 \, dx
=
\frac{N(p-1)-2b}{2(p+1)}
\int_{\R^N} |x|^b |\phi(x)|^{p+1} \, dx,
\end{equation}
and
\begin{equation}
\label{Pohoz5}
\bigl(N+2+2b-(N-2)p\bigr)
\int_{\R^N} |\nabla \phi(x)|^2 \, dx
=
\omega \bigl(N(p-1)-2b\bigr)
\int_{\R^N} |\phi(x)|^2 \, dx.
\end{equation}
In particular, there exists no nontrivial solution to \eqref{GS} when
$$
1+\frac{2b}{N-1} < p < \frac{N+2+2b}{N-2}
\quad \text{and} \quad
\omega \leq 0.
$$

\item The Pohozaev identities link the kinetic, mass, and potential terms for stationary
solutions and yield the constraint $P(\phi)=0$ for every solution of \eqref{GS}.
This is why the functional $P$ is the correct virial quantity in the dynamical analysis.
\end{enumerate}
	\end{remark}

	\section{Proof of Theorem~\ref{G}}
	\label{S3}
	\setcounter{equation}{0}
	In all this section, we suppose that $N\geq 2, b>0, \omega>0$, and $1+\frac{2b}{N-1}<p<p^c$ where $p^c$ is given by~\eqref{pcc}.
	\begin{lemma}
	\label{d}
	We have $d(\omega)>0$.
	\end{lemma}
	\begin{proof}
Let $\phi\in \mathcal{N}_{\omega}$. By definition of the Nehari manifold, we have
\[
\int_{\R^N}|\nabla\phi|^2\,dx+\omega\int_{\R^N}|\phi|^2\,dx
=\int_{\R^N}|x|^b|\phi|^{p+1}\,dx.
\]
In particular, this implies
\begin{equation}
\label{d1}
\|\phi\|_{H^1}^2\lesssim \int_{\R^N}|x|^b|\phi|^{p+1}\,dx.
\end{equation}

On the other hand, by the Gagliardo--Nirenberg inequality (cf. Lemma~\ref{lem-GN-ineq}),
we obtain
\begin{equation}
\label{d2}
\int_{\R^N}|x|^b|\phi|^{p+1}\,dx\lesssim \|\phi\|_{H^1}^{p+1}.
\end{equation}
Combining \eqref{d1} and \eqref{d2}, we infer that
\begin{equation}
\label{d3}
\int_{\R^N}|x|^b|\phi|^{p+1}\,dx\gtrsim 1.
\end{equation}

Since $I_{\omega}(\phi)=0$, it follows that
\[
S_{\omega}(\phi)
=\frac{p-1}{2(p+1)}\int_{\R^N}|x|^b|\phi|^{p+1}\,dx.
\]
Together with \eqref{d3}, this yields $d(\omega)>0$, as claimed.
\end{proof}
We now turn to the proof of Theorem~\ref{G}.

\begin{itemize}
    \item \underline{\bf Proof of Part (i)}\\
    Let $(\phi_n)$ be a minimizing sequence for $d(\omega)$. Then
    \begin{equation}
    \label{Sd}
    S_{\omega}(\phi_n)=\frac{p-1}{2(p+1)}\int_{\R^N}\,|x|^b|\phi_n|^{p+1}\to d(\omega).
\end{equation}
    Hence $\left(\int_{\R^N}\,|x|^b|\phi_n|^{p+1}\right)$ is bounded. Since $I_{\omega}(\phi_n)=0$, we deduce that the sequence $(\phi_n)$
is bounded in $H^1$. It follows that, up to a sub-sequence, we have $\phi_n\rightharpoonup \phi$ weakly in $H^1$. Taking advantage of Lemma~\ref{lem-comp-embe}, we infer
\begin{equation}
\label{phi1}
\lim_{n\to\infty}\int_{\R^N}\,|x|^b|\phi_n|^{p+1}=\int_{\R^N}\,|x|^b|\phi|^{p+1}.
\end{equation}
Recalling \eqref{Sd} we end up with
\begin{equation}
\label{phi2}
\int_{\R^N}\,|x|^b|\phi|^{p+1}=\frac{2(p+1)}{p-1}\,d(\omega)>0.
\end{equation}
In particular, $\phi\neq 0$. To conclude the proof of Part (i), it remains to show that $I_{\omega}(\phi)=0$. It follows from \eqref{phi1} that 
$$
I_{\omega}(\phi)\leq \liminf\,I_{\omega}(\phi_n)=0.
$$
Suppose that $I_{\omega}(\phi)<0$. Hence
$$
\int |\nabla\phi|^2+\omega\int_{\R^N}\,|\phi|^2<\int_{\R^N}\,|x|^b|\phi|^{p+1}.
$$
Set
$$
\lambda=\left(\frac{\|\nabla\phi\|_{L^2}^2+\omega\|\phi\|_{L^2}^2}{\int_{\R^N}\,|x|^b|\phi|^{p+1}}\right)^{\frac{1}{p-1}}\in(0,1).
$$
Clearly $I_{\omega}(\lambda\phi)=0$. Therefore $$
d(\omega)\leq S_{\omega}(\lambda\phi)=\lambda^{p+1} d(\omega)<d(\omega).
$$
This is a contradiction and we conclude that $I_{\omega}(\phi)=0$. Finally 
$$
S_{\omega}(\phi)=\frac{p-1}{2(p+1)}\int_{\R^N}\,|x|^b|\phi|^{p+1}=d(\omega),
$$
and $\phi\in \mathcal{M}_{\omega}.$

\item \underline{\bf Proof of Part (ii)}\\
First let us show that $\mathcal{M}_{\omega}\subset \mathcal{G}_{\omega}.$ Let $\phi\in \mathcal{M}_{\omega}$. There exists a Lagrange multiplier $\lambda$ such that $S'_{\omega}(\phi)=\lambda\, I'_{\omega}(\phi)$. Hence 
$$
0=I_{\omega}(\phi)=\langle S'_{\omega}(\phi), \phi\rangle=\lambda\langle I'_{\omega}(\phi), \phi\rangle=\lambda (1-p)\int_{\R^N}\,|x|^b |\phi|^{p+1}.
$$
Therefore $\lambda=0$ and $\phi\in \mathcal{A}_{\omega}.$ Next, let $v\in \mathcal{A}_{\omega}.$  Then $0=\langle S'_{\omega}(v), v\rangle=I_{\omega}(v).$
It follows that $d(\omega)=S_{\omega}(\phi)\leq S_{\omega}(v)$ and $\phi\in \mathcal{G}_{\omega}$.\\
Now let us prove that $\mathcal{G}_{\omega}\subset \mathcal{M}_{\omega}.$\footnote{We are grateful to {\em Alex H. Ardila} for bringing our attention to this argument in the proof.} Let $\phi\in \mathcal{G}_{\omega}$. Since (see \eqref{GSS}) $\phi \in \mathcal{A}_{\omega}$, we have
$$
I_{\omega}(\phi)=\langle S'_{\omega}(\phi), \phi\rangle=0.
$$
In particular, by \eqref{domega}, $d(\omega)\leq S_{\omega}(\phi).$
Moreover, since $\mathcal{M}_{\omega}\neq\emptyset$, there exists $\chi\in\mathcal{M}_{\omega}$. In particular, $S_{\omega}(\chi)=d(\omega).$ Now as $\mathcal{M}_{\omega}\subset \mathcal{G}_{\omega}$, then $\chi\in \mathcal{G}_{\omega}$. In particular, $\chi \in \mathcal{A}_{\omega}$. Thus, by using the fact that $\phi\in \mathcal{G}_{\omega}$ and \eqref{GSS}, we obtain that
$$
S_{\omega}(\phi)\leq S_{\omega}(\chi)=d(\omega).
$$
Therefore $S_{\omega}(\phi)=d(\omega).$ 
Since $I_{\omega}(\phi)=0$, we see that $\phi\in\mathcal{M}_{\omega} $ Thus, $\mathcal{G}_{\omega}\subset \mathcal{M}_{\omega}$. This concludes the proof of Part (ii).
    \item \underline{\bf Proof of Part (iii)}\\
    Let $\phi\in \mathcal{M}_{\omega}$. We claim that $|\phi|\in \mathcal{M}_{\omega}$.
Indeed, since
\[
\|\nabla|\phi|\|_{L^2}^2\leq \|\nabla\phi\|_{L^2}^2,
\]
it follows that
\begin{eqnarray*}
S_{\omega}(|\phi|)&\leq& S_{\omega}(\phi),\\
I_{\omega}(|\phi|)&\leq& I_{\omega}(\phi)=0.
\end{eqnarray*}
In particular, $I_{\omega}(|\phi|)\leq 0$. Arguing as in the proof of Part~(i),
we infer that $I_{\omega}(|\phi|)=0$, hence $|\phi|\in \mathcal{N}_{\omega}$.

Next, we prove that $S_{\omega}(|\phi|)=d(\omega)$. Clearly,
\[
S_{\omega}(|\phi|)\leq S_{\omega}(\phi)=d(\omega).
\]
On the other hand, since $|\phi|\in \mathcal{N}_{\omega}$, we obtain from
\eqref{domega} that
\[
d(\omega)\leq S_{\omega}(|\phi|).
\]
Therefore $S_{\omega}(|\phi|)=d(\omega)$, and the claim is proved. In particular,
we have
\begin{equation}
\label{nabphi}
\|\nabla|\phi|\|_{L^2}^2 = \|\nabla\phi\|_{L^2}^2.
\end{equation}

Since $\phi$ solves \eqref{GS}, elliptic regularity yields
$\phi\in C^{1}(\R^{N},\C)$. Moreover, since $S_{\omega}(|\phi|)=d(\omega)>0$,
we deduce that $|\phi|>0$ in $\R^N$. Define
\[
\psi(x):=\frac{\phi(x)}{|\phi(x)|}.
\]
Then $|\psi|=1$, so $\psi\bar\psi=1$ and hence
\[
0=\nabla(|\psi|^2)=\nabla(\psi\bar\psi)=2\Re(\bar\psi\nabla\psi),
\]
which implies $\Re(\bar\psi\nabla\psi)=0$. Writing $\phi=|\phi|\psi$, we compute
\[
\nabla\phi=\psi\nabla|\phi|+|\phi|\nabla\psi,
\]
and consequently
\[
|\nabla\phi|^2
=|\nabla|\phi||^2+|\phi|^2|\nabla\psi|^2
+2|\phi|\Re\big(\nabla|\phi|\cdot \bar\psi\nabla\psi\big).
\]
Since
\[
\Re\big(\nabla|\phi|\cdot \bar\psi\nabla\psi\big)
=\nabla|\phi|\cdot \Re(\bar\psi\nabla\psi)=0,
\]
we obtain
\[
|\nabla\phi|^2=|\nabla|\phi||^2+|\phi|^2|\nabla\psi|^2.
\]
Integrating over $\R^N$ and using \eqref{nabphi}, we infer that
\[
\int_{\R^N}|\phi|^2|\nabla\psi|^2\,dx
=\int_{\R^N}\big(|\nabla\phi|^2-|\nabla|\phi||^2\big)\,dx
=0.
\]
Since $|\phi|>0$, it follows that $|\nabla\psi|=0$ a.e. in $\R^N$, and by
continuity $\psi$ must be constant. Hence there exists $\theta\in\R$ such that
$\psi(x)=e^{i\theta}$ for all $x\in\R^N$, and therefore
\[
\phi=e^{i\theta}|\phi|.
\]

Finally, we observe that $|\phi|$ is a positive solution of \eqref{GS}.
Indeed, since $\phi$ satisfies \eqref{GS}, the identity $\phi=e^{i\theta}|\phi|$
implies that $|\phi|$ satisfies
\[
-\Delta|\phi|+\omega|\phi|=|x|^b|\phi|^p.
\]
Moreover, $|\phi|\in \mathcal{M}_{\omega}\cap H^1_{\rm rad}$.
By Proposition~\ref{prop-unique}, we conclude that $|\phi|=Q_\omega$, the unique
positive radial solution of \eqref{GS}. Therefore,
\[
\phi=e^{i\theta}Q_\omega.
\]
This completes the proof of Part~(iii), since $\mathcal{G}_{\omega}=\mathcal{M}_{\omega}$.

\end{itemize}
\section{Proof of Theorem \ref{Kpm}}
\label{S4}
	\setcounter{equation}{0}

For $a,c\in\R$, $\lambda>0$, and $v\in H^1_{\bf r}$, we consider the scaling
\[
v^{\lambda}_{a,c}(x):=\lambda^a\,v\!\left(\frac{x}{\lambda^c}\right).
\]
We then define the associated scaling derivative of the action functional by
\begin{eqnarray*}
K_{a,c}(v)
&:=&\partial_{\lambda}\Big(S_\omega\big(v^{\lambda}_{a,c}\big)\Big)\Big|_{\lambda=1}\\
&=&\frac12(2a+Nc)\|v\|^2+\frac12\big(2a+(N-2)c\big)\|\nabla v\|^2
-\Big(a+\frac{c(b+N)}{1+p}\Big)\int_{\R^N}|x|^b|v|^{1+p}\,dx.
\end{eqnarray*}
We further introduce the constant $B=B(N,b,p)$ defined by
\begin{equation}
\label{B}
B:=\frac{N(p-1)-2b}{2}.
\end{equation}

\noindent $\bullet$ {\bf Stability of $ \mathcal{K}_\omega^\pm$ under the flow of \eqref{INLS}.}\\
Let us show that $\mathcal{K}_\omega^+$ is invariant under the flow of \eqref{INLS}.
Let $u_0\in \mathcal{K}_\omega^+$, that is,
\[
S_\omega(u_0)<d(\omega)
\qquad\text{and}\qquad
P(u_0)\ge 0.
\]
Assume by contradiction that there exists $t_0>0$ such that
$u(t_0)\notin \mathcal{K}_\omega^+$. By the conservation laws and a standard
continuity argument, there then exists a time $t_1\in(0,T^*)$ such that
\[
S_\omega(u(t_1))<d(\omega)
\qquad\text{and}\qquad
P(u(t_1))=0.
\]

Now, arguing as in \cite{imn}, for any pair $(a,c)$ satisfying
\[
\min\Big\{2a+Nc,\;2a+(N-2)c\Big\}\ge 0
\qquad\text{and}\qquad
\max\Big\{2a+Nc,\;2a+(N-2)c\Big\}>0,
\]
one has the variational characterization
\[
d(\omega)=\inf\Big\{S_\omega(v):\ K_{a,c}(v)=0\Big\}.
\]
In particular, since $K_{1,-\frac2N}=\frac2N P$, we obtain
\[
d(\omega)=\inf\Big\{S_\omega(v):\ P(v)=0\Big\}.
\]
Applying this with $v=u(t_1)$ yields
\[
d(\omega)\le S_\omega(u(t_1))=S_\omega(u_0)<d(\omega),
\]
which is a contradiction. This proves that $\mathcal{K}_\omega^+$ is invariant
under the flow of \eqref{INLS}. The invariance of $\mathcal{K}_\omega^-$ is
proved in the same way.\\
\noindent$\bullet$ {\bf Global existence for data in $\mathcal{K}_\omega^+$.}\\
Let $u_0\in \mathcal{K}_\omega^+$. Then, by the invariance of $\mathcal{K}_\omega^+$
under the flow, we have
\[
P(u(t))\ge 0 \qquad \text{for all } t\in[0,T^*).
\]
Recalling the definition of $P$, this yields
\[
\|\nabla u(t)\|_{L^2}^2
\ge \frac{B}{1+p}\int_{\R^N}|x|^b|u(t)|^{1+p}\,dx.
\]
Moreover, using the identity
\[
E(u(t))=\frac12\|\nabla u(t)\|_{L^2}^2-\frac{1}{1+p}\int_{\R^N}|x|^b|u(t)|^{1+p}\,dx,
\]
we infer that
\[
\frac{1}{1+p}\int_{\R^N}|x|^b|u(t)|^{1+p}\,dx
=\frac12\|\nabla u(t)\|_{L^2}^2-E(u(t)).
\]
Consequently,
\[
\|\nabla u(t)\|_{L^2}^2
\ge B\Big(\frac12\|\nabla u(t)\|_{L^2}^2-E(u(t))\Big),
\]
and therefore
\[
\Big(1-\frac{B}{2}\Big)\|\nabla u(t)\|_{L^2}^2\ge -B\,E(u(t)).
\]
Since $E(u(t))=E(u_0)$ is conserved, we obtain the uniform bound
\[
\|\nabla u(t)\|_{L^2}^2\le \frac{2B}{B-2}\,E(u_0).
\]
In particular, $\|u(t)\|_{H^1}$ remains bounded on $[0,T^*)$, and thus the
maximal existence time satisfies $T^*=\infty$.
\\
\noindent$\bullet$ {\bf Non-global existence for data in $\mathcal{K}_\omega^-$.}\\
Let us start with the following auxiliary characterization of $d(\omega)$.
\begin{lemma}
One has
\[
d(\omega)=\inf_{0\neq u\in H^1}\Big\{\big(S_\omega-\tfrac12P\big)(u):\ P(u)\leq0\Big\}.
\]
\end{lemma}

\begin{proof}
Let $m(\omega)$ denote the right-hand side. It suffices to prove that
$d(\omega)\leq m(\omega)$.

Fix $u\in H^1$ such that $P(u)<0$ and consider the $L^2$-invariant scaling
$u_\lambda(x):=\lambda^{\frac N2}u(\lambda x)$, $\lambda>0$. A direct computation yields
\begin{align*}
P(u_\lambda)
&=\lambda^2\Big(\|\nabla u\|^2-\frac{N(p-1)-2b}{2(p+1)}\lambda^{B-2}
\int_{\R^N}|x|^b|u|^{p+1}\,dx\Big),\\
\big(S_\omega-\tfrac12P\big)(u_\lambda)
&=\frac\omega2\|u\|^2+\frac{N(p-p_c)}{4(1+p)}\lambda^B
\int_{\R^N}|x|^b|u|^{p+1}\,dx.
\end{align*}
Since $B>2$, we have $P(u_\lambda)>0$ for $0<\lambda\ll 1$. By continuity of
$\lambda\mapsto P(u_\lambda)$, there exists $\lambda_0\in(0,1)$ such that
$P(u_{\lambda_0})=0$. Moreover, $\lambda\mapsto \big(S_\omega-\tfrac12P\big)(u_\lambda)$
is increasing, hence
\[
d(\omega)\le \big(S_\omega-\tfrac12P\big)(u_{\lambda_0})
\le \big(S_\omega-\tfrac12P\big)(u).
\]
Taking the infimum over all $u$ with $P(u)\le 0$ gives $d(\omega)\le m(\omega)$,
and the proof is complete.
\end{proof}

We now prove the last part of Theorem~\ref{Kpm}. Let
$u_0\in \mathcal{K}_\omega^-\cap L^2(|x|^2\,dx)$ and assume by contradiction that
the corresponding solution to \eqref{INLS} is global. Then, by invariance,
\[
u(t)\in \mathcal{K}_\omega^-\cap L^2(|x|^2\,dx)
\qquad\text{for all } t\ge 0.
\]
We claim that there exists $\eta>0$ such that
\[
P(u(t))\le -\eta
\qquad\text{for all sufficiently large } t.
\]
Otherwise, one can find a sequence $t_n\to\infty$ such that $P(u(t_n))\to 0$.
Applying the previous lemma, we obtain
\[
m(\omega)\le \big(S_\omega-\tfrac12P\big)(u(t_n))
= S_\omega(u_0)-\tfrac12P(u(t_n))
\longrightarrow S_\omega(u_0).
\]
Since $u_0\in \mathcal{K}_\omega^-$, we have $S_\omega(u_0)<d(\omega)=m(\omega)$,
and therefore the above limit implies
\[
m(\omega)\le S_\omega(u_0)<m(\omega),
\]
which is a contradiction. This concludes the proof.

\section{Proof of Theorem~\ref{Stab-Instab-GS}}
\label{S5}
\setcounter{equation}{0}

In this section, we prove the stability and instability statements for ground state
standing waves. We begin with a regularity property of the threshold level $d(\omega)$,
which will be used in the stability argument.

\begin{lemma}
\label{Stab-Lem}
The function $\omega \longmapsto d(\omega)$ is $C^\infty(0,\infty)$. Moreover, if
$p< 1+\frac{4+2b}{N}$, then $d''(\omega)>0$ for all $\omega>0$.
\end{lemma}

\begin{proof}
By a scaling argument, we have
\[
\phi_{\omega}(x)=\omega^{\frac{2+b}{2(p-1)}}\phi_1(\sqrt{\omega}\,x),
\]
where $\phi_1$ denotes a ground state of \eqref{GS} corresponding to $\omega=1$.
Consequently,
\[
d(\omega)=S_{\omega}(\phi_{\omega})
=\frac{p-1}{2(p+1)}\, \omega^{\frac{(2+b)(p+1)}{2(p-1)}-\frac{b+N}{2}}
\int_{\R^N}|x|^b|\phi_1(x)|^{p+1}\,dx.
\]
This explicit expression shows that $d\in C^\infty(0,\infty)$. Furthermore, if
$p< 1+\frac{4+2b}{N}$, then
\[
\frac{(2+b)(p+1)}{2(p-1)}-\frac{b+N}{2}>1,
\]
and therefore $d''(\omega)>0$ for all $\omega>0$. This completes the proof.
\end{proof}

We next establish a key deformation property of ground states, which will be used
to prove strong instability in the mass-supercritical case.

\begin{lemma}
\label{Instab-Lem}
Let $\omega >0$, $b>0$, $1+\frac{4+2b}{N}< p<\frac{N+2+2b}{N-2}$, and
$\phi\in \mathcal{G}_{\omega}$. For $\lambda>0$, define
\[
\phi^{\lambda}(x)={\rm e}^{\frac{N\lambda}{2}}\phi({\rm e}^{\lambda}\,x).
\]
Then
\[
\phi^{\lambda}\in \mathcal{K}_{\omega}^{-}\cap \Sigma.
\]
\end{lemma}

\begin{proof}
Clearly $\phi^{\lambda}\in \Sigma$ since $\phi\in \Sigma$. A direct computation gives
\[
S_{\omega}(\phi^{\lambda})
=\frac{{\rm e}^{2\lambda}}{2}\|\nabla\phi\|_{L^2}^2+\frac{\omega}{2}\|\phi\|_{L^2}^2
-\frac{{\rm e}^{\frac{\lambda}{2}(N(p-1)-2b)}}{p+1}
\int_{\R^N}|x|^b|\phi(x)|^{p+1}\,dx.
\]
Using $P(\phi)=0$, we obtain
\begin{equation}
\label{S-lambda}
\partial_{\lambda}\,S_{\omega}(\phi^{\lambda})
={\rm e}^{2\lambda}\bigg(1-{\rm e}^{\frac{\lambda}{2}(Np-(N+4+2b))}\bigg)\,
\|\nabla\phi\|_{L^2}^2.
\end{equation}
Since $1+\frac{4+2b}{N}<p$, we have $Np-(N+4+2b)>0$, and thus
$\partial_{\lambda}\,S_{\omega}(\phi^{\lambda})<0$ for all $\lambda>0$.
Consequently,
\[
S_{\omega}(\phi^{\lambda})<S_{\omega}(\phi^{0})=S_{\omega}(\phi)
\qquad\text{for all }\lambda>0.
\]
Finally, observing that $\partial_{\lambda}S_{\omega}(\phi^{\lambda})=P(\phi^{\lambda})$,
we conclude that $P(\phi^{\lambda})<0$ for all $\lambda>0$, and therefore
$\phi^{\lambda}\in \mathcal{K}_{\omega}^{-}\cap \Sigma$.
\end{proof}

To treat the mass-critical case, namely $p=1+\frac{4+2b}{N}$, we use the following lemma.

\begin{lemma}
\label{Mass-Crit}
Let $b>0$, $\omega>0$, $p=1+\frac{4+2b}{N}$, and $\phi\in\mathcal{G}_{\omega}$.
Let $(\lambda_n)$ be a sequence in $(1,\infty)$ such that $\lim\lambda_n=1$, and define
\[
\phi_n(x)=\lambda_n^{1+N/2}\,\phi(\lambda_n\,x).
\]
Then
\begin{equation}
\label{Mass-Crit1}
\|\phi_n-\phi\|_{H^1}\to 0\quad\mbox{as}\quad n\to\infty,
\end{equation}
and
\begin{equation}
\label{Mass-Crit2}
E(\phi_n)<0, \;\;\;\forall\;n\geq 1.
\end{equation}
\end{lemma}

\begin{proof}
A direct computation yields
\begin{eqnarray}
\label{L2}
\|\phi_n\|_{L^2}&=&\lambda_n\,\|\phi\|_{L^2},\\
\label{nabla-L2}
\|\nabla\phi_n\|_{L^2}&=&\lambda_n^2\,\|\nabla\phi\|_{L^2},\\
\label{x-phi-n}
\int_{\R^N}|x|^b|\phi_n(x)|^{p+1}\,dx&=&
\lambda_n^{\frac{(N+2)(p+1)}{2}-b-N}
\int_{\R^N}|x|^b|\phi(x)|^{p+1}\,dx.
\end{eqnarray}
In particular, \eqref{Mass-Crit1} follows from Brezis--Lieb's lemma \cite{Bre-Lieb}.

To prove \eqref{Mass-Crit2}, we use \eqref{nabla-L2}--\eqref{x-phi-n} to obtain
\begin{equation}
\label{E-phi-n}
E(\phi_n)
=\frac{\lambda_n^4}{2}\|\nabla\phi\|_{L^2}^2
-\frac{\lambda_n^{\frac{(N+2)(p+1)}{2}-b-N}}{p+1}
\int_{\R^N}|x|^b|\phi(x)|^{p+1}\,dx.
\end{equation}
Using \eqref{Pohoz4} and the fact that $p=1+\frac{4+2b}{N}$, we can rewrite
\eqref{E-phi-n} as
\begin{equation}
\label{E-phi-nn}
E(\phi_n)
=\frac{\lambda_n^4}{p+1}\left(1-\lambda_n^{\frac{4+2b}{N}}\right)
\int_{\R^N}|x|^b|\phi(x)|^{p+1}\,dx.
\end{equation}
Since $\lambda_n>1$, this implies $E(\phi_n)<0$ for all $n\ge 1$, proving \eqref{Mass-Crit2}.
\end{proof}

We are now in a position to complete the proof of Theorem~\ref{Stab-Instab-GS}.

\begin{itemize}
\item \underline{\bf Proof of Part (i).}\\
Let $\phi\in G_{\omega}$. By Lemma~\ref{Instab-Lem}, we have
$\phi^{\lambda}\in \mathcal{K}_{\omega}^-\cap\Sigma$ for every $\lambda>0$.
Therefore, by Part~(iii) of Theorem~\ref{Kpm}, the corresponding solution to
\eqref{INLS} with initial datum $\phi^{\lambda}$ blows up in finite time.
Since $\|\phi^{\lambda}-\phi\|_{H^1}\to 0$ as $\lambda\to 0$, we conclude that the
standing wave ${\rm e}^{i\omega t}\phi(x)$ is strongly unstable.

\item \underline{\bf Proof of Part (ii).}\\
This follows by combining Lemma~\ref{Mass-Crit} with \cite[Theorem~1.3]{DMS2021}.

\item \underline{\bf Proof of Part (iii).}\\
The proof relies on Lemma~\ref{Stab-Lem} together with an adaptation of the
variational argument in \cite[Proof of Theorem~1.5(i), p.~117]{CPAM2021}.
We omit the details.
\end{itemize}
    \section{Proof of Theorem~\ref{Norm-SW-Subcrit}}
\label{S6}
\setcounter{equation}{0}

We begin with the proof of Proposition~\ref{scal}, which provides a scaling
construction of normalized solutions to \eqref{GS}.

\begin{proof}
Let $Q$ be the unique positive radial solution to \eqref{GS} with $\omega=1$
(see~\cite[Proposition 2.3]{DMS2021}). For $\lambda>0$ (to be fixed later), define
\[
\phi(x)=\lambda^{\frac{2+b}{p-1}}\,Q(\lambda x).
\]
A straightforward computation shows that $\phi$ solves
\[
-\Delta\phi+\lambda^2\phi=|x|^b\phi^p,
\]
and moreover
\[
\|\phi\|_{L^2}^2=\lambda^{\frac{N}{p-1}(p_c-p)}\,\|Q\|_{L^2}^2.
\]
Choosing
\[
\lambda:=\left(c\|Q\|_{L^2}^{-2}\right)^{\frac{p-1}{N(p_c-p)}},
\]
we obtain a function
\[
\phi_c(x)=\lambda^{\frac{2+b}{p-1}}\,Q(\lambda x)
\]
which solves \eqref{GS} with
\[
\omega=\omega_c=\left(c\|Q\|_{L^2}^{-2}\right)^{\frac{2(p-1)}{N(p_c-p)}},
\qquad\text{and satisfies}\qquad
\|\phi_c\|_{L^2}^2=c.
\]
The remaining assertions follow immediately.
\end{proof}

Throughout the rest of this section, we assume that
\[
b>0,\qquad N>1+\frac{b}{2},\qquad c>0,\qquad 1+\frac{2b}{N-1}<p<p_c.
\]

\begin{itemize}
\item \underline{\bf Proof of Part (i) of Theorem~\ref{Norm-SW-Subcrit}.}\\
The proof is divided into two steps.

\medskip
\noindent\textbf{Step 1.} We claim that
\begin{equation}
\label{m-c-0}
-\infty<\mathbf{m}(c)<0.
\end{equation}
Let $\phi\in \mathcal{S}(c)$ and $\mu>0$, and set $\phi_{\mu}(x)=\mu^{\frac{N}{2}}\phi(\mu x)$.
Then $\phi_{\mu}\in \mathcal{S}(c)$ and
\[
E(\phi_{\mu})
=\mu^2\bigg(\frac{1}{2}\|\nabla\phi\|_{L^2}^2
-\frac{\mu^{\kappa}}{p+1}\int_{\R^N}|x|^b|\phi(x)|^{p+1}\,dx\bigg),
\]
where $\kappa:=\frac{N}{2}(p+1)-N-b-2=\frac N2(p-p_c)<0$ since $p<p_c$.
Hence $E(\phi_{\mu})<0$ for $\mu>0$ small enough, and therefore $\mathbf{m}(c)<0$.

Next, using the Gagliardo--Nirenberg inequality \eqref{GNI}, we have for $\phi\in \mathcal{S}(c)$
\begin{eqnarray*}
\frac{1}{p+1}\int_{\R^N}|x|^b|\phi(x)|^{p+1}\,dx
&\lesssim& \|\phi\|_{L^2}^{\frac{4+2b-(N-2)(p-1)}{2}}\,
\|\nabla\phi\|_{L^2}^{\frac{N(p-1)-2b}{2}}\\
&\lesssim& \|\nabla\phi\|_{L^2}^{\frac{N(p-1)-2b}{2}}
\le \frac{1}{4}\|\nabla\phi\|_{L^2}^2+K,
\end{eqnarray*}
where $0<K<\infty$ depends only on $b,c,N,p$. Here we used Young's inequality
\[
xy\leq \frac{x^q}{4}+\left(1-\frac{1}{q}\right)\left(\frac{q}{4}\right)^{\frac{1}{1-q}}\,
y^{\frac{q}{q-1}},\qquad x,y\ge 0,\quad q>1,
\]
together with the assumption $p<1+\frac{4+2b}{N}$. It follows that
\[
E(\phi)\ge \frac{1}{4}\|\nabla\phi\|_{L^2}^2-K\ge -K,
\]
for all $\phi\in \mathcal{S}(c)$, and thus $\mathbf{m}(c)>-\infty$.
This proves \eqref{m-c-0}.

\medskip
\noindent\textbf{Step 2.} \emph{Existence of a minimizer for $\mathbf{m}(c)$.}
Let $(\phi_n)\subset H^1_{\rm rad}(\mathbb R^N)$ be a minimizing sequence for $\mathbf{m}(c)$, namely
\begin{equation}\label{6.1}
\phi_n\in S(c), \qquad E(\phi_n)\to \mathbf{m}(c).
\end{equation}
By Step~1, $(\phi_n)$ is bounded in $H^1_{\rm rad}(\mathbb R^N)$. Hence, up to a subsequence,
\[
\phi_n\rightharpoonup \phi \quad \text{weakly in } H^1(\mathbb R^N),
\qquad
\phi_n(x)\to \phi(x) \quad \text{a.e. in } \mathbb R^N.
\]
Moreover, by Lemma~\ref{lem-comp-embe}, the embedding
$H^1_{\rm rad}(\mathbb R^N)\hookrightarrow L^{p+1}(\mathbb R^N,|x|^b\,dx)$
is compact. Therefore,
\begin{equation}\label{6.2}
\phi_n \to \phi \quad \text{strongly in } L^{p+1}(\mathbb R^N,|x|^b\,dx),
\end{equation}
and in particular
\[
\int_{\mathbb R^N}|x|^b|\phi_n|^{p+1}\,dx
\longrightarrow
\int_{\mathbb R^N}|x|^b|\phi|^{p+1}\,dx .
\]

We now show that $\|\phi\|_{L^2}^2=c$. By weak lower semicontinuity,
\[
\|\phi\|_{L^2}^2 \le \liminf_{n\to\infty}\|\phi_n\|_{L^2}^2 = c.
\]
Assume by contradiction that $\|\phi\|_{L^2}^2=:c_1<c$. Then $\phi\in S(c_1)$ and hence
$E(\phi)\ge \mathbf{m}(c_1)$.

We claim that $\mathbf{m}(c)$ is strictly decreasing in $c$. Indeed, let $0<c_1<c_2$ and
$u\in S(c_1)$. Set $\alpha=\sqrt{c_2/c_1}>1$, so that $\alpha u\in S(c_2)$. Since $p+1>2$,
\[
E(\alpha u)
= \frac{\alpha^2}{2}\|\nabla u\|_{L^2}^2
- \frac{\alpha^{p+1}}{p+1}\int_{\mathbb R^N}|x|^b|u|^{p+1}\,dx
< \alpha^2 E(u).
\]
Taking the infimum over $u\in S(c_1)$ yields
\[
\mathbf{m}(c_2) < \alpha^2 \mathbf{m}(c_1).
\]
Since $\mathbf{m}(c_1)<0$ by Step~1 and $\alpha^2>1$, it follows that
$\mathbf{m}(c_2)<\mathbf{m}(c_1)$.

Applying this with $c_1<c$, we obtain $\mathbf{m}(c)<\mathbf{m}(c_1)$.
On the other hand, by \eqref{6.2} and weak lower semicontinuity of the gradient norm,
\[
\|\nabla\phi\|_{L^2}^2 \le \liminf_{n\to\infty}\|\nabla\phi_n\|_{L^2}^2,
\]
and therefore
\[
E(\phi)\le \liminf_{n\to\infty}E(\phi_n)=\mathbf{m}(c).
\]
Hence,
\[
\mathbf{m}(c)\ge E(\phi)\ge \mathbf{m}(c_1),
\]
which contradicts $\mathbf{m}(c)<\mathbf{m}(c_1)$. Thus $\|\phi\|_{L^2}^2=c$.

Consequently, $\phi\in S(c)$ and
\[
\mathbf{m}(c)\le E(\phi)\le \liminf_{n\to\infty}E(\phi_n)=\mathbf{m}(c),
\]
so $E(\phi)=\mathbf{m}(c)$ and $\phi\in \mathbb{M}_c$.

Finally, since $\|\phi_n\|_{L^2}\to\|\phi\|_{L^2}$ and $\phi_n\rightharpoonup\phi$ in $L^2$,
we have $\phi_n\to\phi$ strongly in $L^2(\mathbb R^N)$. Using \eqref{6.2} and
$E(\phi_n)\to E(\phi)$, we also get
\[
\|\nabla\phi_n\|_{L^2}^2
= 2E(\phi_n)+\frac{2}{p+1}\int_{\mathbb R^N}|x|^b|\phi_n|^{p+1}\,dx
\longrightarrow \|\nabla\phi\|_{L^2}^2,
\]
hence $\phi_n\to\phi$ strongly in $H^1(\mathbb R^N)$. This completes the proof of
Part~(i) of Theorem~\ref{Norm-SW-Subcrit}.

\medskip
\noindent\textbf{Proof of Part (ii) of Theorem~\ref{Norm-SW-Subcrit}.}\\
Assume, by contradiction, that $\mathbb{M}_c$ is not orbitally stable in the sense of
Definition~\ref{Orb-Stab-NSW}. Then there exist $\varepsilon_0>0$, a sequence of initial data
$(u_{0,n})\subset H^1(\mathbb R^N)$, and a sequence of times $t_n>0$ such that
\begin{equation}\label{6.4}
\inf_{\phi\in \mathbb{M}_c}\|u_{0,n}-\phi\|_{H^1}\longrightarrow 0
\quad \text{as } n\to\infty,
\end{equation}
but
\begin{equation}\label{6.5}
\inf_{\phi\in \mathbb{M}_c}\|u_n(t_n)-\phi\|_{H^1}\ge \varepsilon_0
\quad \text{for all } n,
\end{equation}
where $u_n(t)$ denotes the solution of \eqref{INLS} with initial data $u_{0,n}$.

By conservation of mass and energy along the flow of \eqref{INLS}, we have
\[
\|u_n(t_n)\|_{L^2}^2=\|u_{0,n}\|_{L^2}^2
\quad \text{and} \quad
E(u_n(t_n))=E(u_{0,n}).
\]
From \eqref{6.4} and the fact that every element of $\mathbb{M}_c$ has mass $c$ and energy
$\mathbf{m}(c)$, it follows that
\begin{equation}\label{6.6}
\|u_n(t_n)\|_{L^2}^2 \longrightarrow c,
\qquad
E(u_n(t_n)) \longrightarrow \mathbf{m}(c).
\end{equation}

Define the normalization factor
\[
\alpha_n := \sqrt{\frac{c}{\|u_n(t_n)\|_{L^2}^2}},
\]
and set
\[
v_n := \alpha_n\,u_n(t_n).
\]
Then $\|v_n\|_{L^2}^2=c$ for all $n$, and by \eqref{6.6} we have $\alpha_n\to 1$ as
$n\to\infty$. Moreover, since the map $\alpha\mapsto E(\alpha u)$ is continuous for fixed
$u\in H^1(\mathbb R^N)$, we obtain
\[
E(v_n)-E(u_n(t_n))\longrightarrow 0.
\]
Combining this with \eqref{6.6}, we conclude that
\[
E(v_n)\longrightarrow \mathbf{m}(c),
\]
so that $(v_n)\subset S(c)$ is a minimizing sequence for $\mathbf{m}(c)$.

By Part~(i) of Theorem~\ref{Norm-SW-Subcrit}, there exist a subsequence (still denoted $(v_n)$)
and $\phi\in \mathbb{M}_c$ such that
\[
v_n \longrightarrow \phi \quad \text{strongly in } H^1(\mathbb R^N).
\]
Since $\alpha_n\to 1$, it follows that
\[
u_n(t_n)=\alpha_n^{-1}v_n \longrightarrow \phi
\quad \text{strongly in } H^1(\mathbb R^N).
\]
Hence,
\[
\inf_{\psi\in \mathbb{M}_c}\|u_n(t_n)-\psi\|_{H^1}
\le \|u_n(t_n)-\phi\|_{H^1}\longrightarrow 0,
\]
which contradicts \eqref{6.5}. Therefore $\mathbb{M}_c$ is orbitally stable, and the proof
of Part~(ii) of Theorem~\ref{Norm-SW-Subcrit} is complete.
\end{itemize}
\section{Concluding Remarks}
\label{sec:concl}
In this paper we studied the focusing inhomogeneous nonlinear Schrödinger equation
\eqref{INLS} with spatially growing nonlinearity ($b>0$), with particular emphasis on
the variational construction of ground state standing waves and their role as
threshold objects for the dynamics. A central difficulty in the case $b>0$ is that
the weight $|x|^b$ destroys the usual translation-invariant structure of the
homogeneous NLS, and this substantially modifies both the variational framework and
the dynamical analysis. In the radial setting, however, the growth of the weight at
infinity yields a compactness mechanism which is crucial in the minimization
procedures and in the characterization of ground states.

At fixed frequency $\omega>0$, we established existence and qualitative properties
of ground states in $\mathcal{G}_\omega$, and we derived a sharp dichotomy between
global existence and finite-time blow-up in the spirit of the potential-well theory.
More precisely, the sets $\mathcal{K}_\omega^{+}$ and $\mathcal{K}_\omega^{-}$ provide
a natural partition of the phase space below the ground state level $d(\omega)$,
leading respectively to global existence or blow-up under the corresponding
assumptions. This yields a coherent description of the dynamics near the ground
state manifold, and highlights the decisive influence of the exponent $p$ through
the mass-subcritical, mass-critical, and mass-supercritical regimes.

We also investigated normalized standing waves, that is, standing waves with a
prescribed mass $c>0$. In the mass-subcritical range, we proved that the constrained
variational problem at fixed mass admits minimizers and that the associated set
$\mathbb{M}_c$ is orbitally stable. The proof relies on a constrained compactness
argument for minimizing sequences in the radial energy space, together with a
standard stability-by-contradiction strategy based on the conservation laws of
\eqref{INLS}. In contrast, in the mass-supercritical regime, the energy becomes
unbounded from below on the $L^2$-sphere, which forces one to adopt a modified
constraint in order to recover meaningful variational information.

Several natural directions remain open. It would be interesting to understand the
endpoint and limiting regimes where compactness deteriorates, as well as to extend
the analysis beyond the radial framework, where the lack of compactness becomes more
severe. Another challenging question concerns normalized standing waves in the
mass-critical and mass-supercritical cases, where the constrained minimization
problem requires additional structure and where the stability picture is expected to
be more subtle. Finally, the approach developed here suggests that similar questions
could be addressed for more general inhomogeneous coefficients $K(x)$, and it would
be worthwhile to investigate how different growth or decay behaviors at infinity
affect the existence and dynamical properties of standing waves.




\end{document}